\documentclass[11pt]{article}

\usepackage{fullpage,amsthm,amsmath,amssymb,amsbsy,amsfonts}

\usepackage{graphicx}

\title{\bf Sharp pointwise estimates for solutions of weakly coupled second order parabolic 
system in a layer}
 
\author{\sc{Gershon Kresin$^a\!\!$}
\thanks{Corresponding author. E-mail: kresin@ariel.ac.il}$\;\;$
 and \sc{Vladimir Maz'ya$^b$}
\thanks{E-mail: vladimir.mazya@liu.se}$\;\;$ 
\\ \\
{\it{$^a$Department of Mathematics, Ariel University, Ariel 40700, Israel}}\\
{\it{$^b$Department of Mathematical Sciences, University of Liverpool,
M$\&$O Building, Liverpool,}}\\ 
{\it{L69 3BX, UK; Department of Mathematics, Link\"oping University,SE-58183 Link\"oping, }}\\
{\it{{\hskip -19mm}Sweden; }}
{\it{RUDN University, 6 Miklukho-Maklay St., Moscow, 117198, Russia}}\\
}
{ \date\ }

\numberwithin{equation}{section}
\newtheorem{lemma}{Lemma}
\newtheorem{theorem}{Theorem}
\newtheorem{proposition}[theorem]{Proposition}

\newenvironment{remark}{{\bf Remark}}

\newcommand{\bs}{\boldsymbol}
\newcommand{\bv}{\bs |\mskip-5mu \bs |\mskip-5mu \bs |}  
\newcommand{\nl}{\lVert}
\newcommand{\nr}{\rVert}

\begin{document}
\maketitle
\centerline{\sl\Large In memory of great mathematician S.L. Sobolev}
\large
\vspace{10mm}

{\bf Abstract.} We deal with $m$-component vector-valued solutions to the Cauchy problem for 
linear both homogeneous and nonhomogeneous weakly coupled second order parabolic system in the 
layer ${\mathbb R}^{n+1}_T={\mathbb R}^n\times (0, T)$. We assume that coefficients of the 
system are real and depending only on $t$, $n\geq 1$ and $T<\infty$. 
The homogeneous system is considered with initial data in $[L^p({\mathbb R}^n)]^m$, $1\leq p \leq \infty $. 
For the nonhomogeneous system we suppose that the initial function is equal to zero and the right-hand side belongs to $[L^p({\mathbb R}^{n+1}_T)]^m\cap [C^\alpha \big (\overline{{\mathbb R}^{n+1}_T} \big )]^m $, $\alpha \in (0, 1)$. Explicit formulas for the sharp coefficients in pointwise estimates for solutions of these problems and their directional derivative are obtained.
\\
\section{Introduction}\label{S_1}

Parabolic equations and systems are classical subjects of mathematical physics (e.g. \cite{EID}, \cite{FR}, \cite{SLS}, \cite{TS}). The present paper is a continuation of our recent work  \cite{KM4} on  sharp pointwise estimates for the gradient of solutions to the Cauchy problem for the single parabolic equation of the second order with constant coefficients. We say that the estimate  is sharp if the coefficient in front of the norm in the  majorant part of the inequality cannot be diminished.
Sharp pointwise estimates for solutions to the Laplace, modified Helmholtz, Lam\'e, Stokes, and heat equations , as well as for the analytic functions  were obtained earlier in \cite {KM} - \cite {KBY}. 

In this paper we study solutions of the Cauchy problem for parabolic weakly coupled system with real coefficients of the form
\begin{equation} \label{PWCS}
\frac{\partial \bs u}{\partial t}=
\sum_{j,k=1}^n a_{jk}(t)\frac{\partial^2 \bs u}{\partial x_j \partial x_k}+\sum_{j=1}^n b_{j}(t)\frac{\partial \bs u}{\partial x_j}+C(t)\bs u+\bs f(x, t)
\end{equation}
in the layer ${\mathbb R}^{n+1}_T={\mathbb R}^n\times (0, T )$ with initial condition $\bs u|_{t=0} =\bs \varphi$, considering separately two cases, $\bs f=\bs 0$ and $\bs \varphi=\bs 0 $. Here and henceforth $T<\infty$, $n\geq 1$,   
$\bs u(x, t)=(u_1(x, t),\dots ,u_m(x, t))$ and $\bs f(x, t)=(f_1(x, t),\dots ,f_m(x, t))$. 

Throughout the article, we assume that $A(t)=((a_{jk}(t)))$ is a symmetric positive definite $(n\times n)$-matrix-valued function on $[0, T]$, which elements satisfy the H\"older condition with exponent $\alpha /2$ ($0< \alpha <1$),
$b_1(t),\dots,b_n(t)$ are continuous functions on $[0, T]$, $C(t)$
is continuous on $[0, T]$ matrix-valued function of order $m$.  
 
We obtain sharp pointwise estimates for $|\bs u|$, $|\partial \bs u/\partial \bs\ell|$ and 
$\max_{|\bs\ell|=1}|\partial \bs u/\partial \bs\ell|$,
where $\bs u$ solves the Cauchy problem for system (\ref{PWCS}) in the layer ${\mathbb R}^{n+1}_T$, $|\cdot |$ denotes the Euclidean length of a vector and $\bs\ell $ is a unit $n$-dimensional vector. 
By $\partial \bs u/\partial \bs\ell$ we mean the derivative of a vector-valued function ${\bs u}(x, t)$ in the direction $\bs\ell $: 
\begin{eqnarray} \label{Eq_3.1AAB}
\frac{\partial {{\bs u}}}{ \partial {\bs\ell}}&=&\lim_{\lambda
\rightarrow 0+}\frac{{\bs u}(x+\lambda {\bs \ell}, t)-{\bs u}(x, t) }{ \lambda}\nonumber \\
&=&({\bs \ell}, \nabla_x ){{\bs u}}=\sum _{j=1}^m \frac{\partial u_j}{\partial \bs \ell}\;{\bs e}_j\;,
\end{eqnarray}
where $\nabla_x=(\partial/\partial x_1,\dots , \partial/\partial x_n )$ and $\bs e_j$ means the unit vector of the $j$-th coordinate axis. 

The present paper consists of five sections,
including Introduction. Section \ref{S_2} is auxiliary. 

Section \ref{S_3} is devoted to explicit formulas for solutions of the Cauchy problem in ${\mathbb R}^{n+1}_T$ for system (\ref{PWCS}).

In Section \ref{S_4}, we consider a solution of the Cauchy problem 
\begin{eqnarray} \label{H}
\left\{\begin{array}{ll}
\displaystyle{\frac{\partial \bs u}{\partial t}= \sum_{j,k=1}^n a_{jk}(t)\frac{\partial^2 \bs u}{\partial x_j \partial x_k}}+\sum_{j=1}^n b_{j}(t)\frac{\partial \bs u}{\partial x_j}
+C(t) \bs u& \quad{\rm in}\; {\mathbb R}^{n+1}_T, \\
       \\
\displaystyle{\bs  u\big |_{t=0}=\bs \varphi }\;,  
\end{array}\right .
\end{eqnarray}
where $\bs\varphi\in [L^p({\mathbb R}^n)]^m$, $p \in [1, \infty]$.
The norm $\nl \cdot \nr_p$ in the space $[L^p({\mathbb R}^n)]^m$ is defined by
$$
\nl \bs\varphi\nr_{p}=\left \{\int _ {{\mathbb R}^n} | \bs\varphi(x) |^p dx \right \}^{1/p}
$$
for $1\leq  p< \infty $, and 
$$
\nl \bs\varphi \nr_{\infty} =\mbox{ess}\;\sup \{ | \bs\varphi(x) |: x \in {{\mathbb R}^n } \}\;. 
$$ 
In this section we obtain two groups of sharp estimates. 
First of them concerns modulus to solution $\bs u$ of problem (\ref{H}). Namely, we derive the inequality 
\begin{equation} \label{Eq_3.007}
|\bs u(x, t) |\leq {\mathcal H}_p(t) \nl \bs\varphi \nr_{p}
\end{equation}
with the sharp coefficient 
\begin{equation} \label{Eq_3.007A}
{\mathcal H}_p(t)=\frac{\bv e^{{\mathcal I}_{C^*}(t)}\bv }{(2\sqrt{\pi })^{n/p}\left (\det {\mathcal I}_A^{1/2}(t)\right )^{1/p}(p')^{n/(2p')}}\;,
\end{equation}
where $(x, t)$ is an arbitrary point in the layer ${\mathbb R}^{n+1}_T$. Here and henceforth ${\mathcal I}_F(t)=\int_0^t F(t)dt$, where $F$ can be a vector-valued or matrix-valued function, $p^{-1}+p'^{-1}=1$, the symbol $^*$ denotes passage to the transposed matrix and $\bv B \bv=\max_{|\bs z|=1}|B\bs z |$ means the spectral norm of the $l\times l$ real-valued matrix $B$, $\bs z \in {\mathbb R}^l$. It is known (e.g. \cite{LA}, sect. 6.3) that $\bv B \bv=\lambda ^{1/2}_B$, where $\lambda _B$ is the spectral radius of the matrix $B^*B$.

As a special case of (\ref{Eq_3.007A}) one has
\begin{equation} \label{Eq_3.007B}
{\mathcal H}_\infty (t)=\bv e^{{\mathcal I}_{C^*}(t)}\bv \;.
\end{equation}
For the single parabolic equation 
\begin{equation} \label{Eq_SHE}
\frac{\partial u}{\partial t}= \sum_{j,k=1}^n a_{jk}(t)\frac{\partial^2 u}{\partial x_j \partial x_k}+\sum_{j=1}^n b_{j}(t)\frac{\partial  u}{\partial x_j}
+c(t)u+f(x, t)
\end{equation}
with $f= 0$, formula (\ref{Eq_3.007B}) becomes
$$
{\mathcal H}_\infty (t)=\exp\left \{ \int_0^t c(t)dt \right \}\;.
$$
The second group of sharp estimates concerns modulus of
$\partial \bs u/ \partial \bs\ell $, where $\bs u$ is solution  of problem (\ref{H}). 
Namely, the formula for the sharp coefficient 
\begin{equation} \label{Eq_1.3P}
{\mathcal K}_{p,\bs\ell}(t)=\frac{
\big |{\mathcal I}_A^{-1/2}(t) \bs\ell \big | \bv \;e^{{\mathcal I}_{C^*}(t)}\bv }
{\big \{ 2^{n}\pi^{(n+p-1)/2}\det{\mathcal I}_A^{1/2}(t) \big \}^{1/p}}
\left \{ \frac{\Gamma \left (\frac{p'+1}{2} \right )} {p'^{(n+p')/2}} \right \}^{1/p'}
\end{equation}
in the inequality 
\begin{equation} \label{Eq_1.3A}
\left | \frac{\partial \bs u}{ \partial \bs\ell }(x, t) \right |\leq {\mathcal K}_{p, \bs\ell}(t) \nl \bs\varphi \nr_{p}
\end{equation}
is obtained, where $(x, t)$ is an arbitrary point in the layer ${\mathbb R}^{n+1}_T$. 

As a consequence of (\ref{Eq_1.3P}), the sharp coefficient 
\begin{equation} \label{Eq_1.3AB}
{\mathcal K}_p(t)=\max_{|\bs\ell |=1}{\mathcal K}_{p, \bs\ell}(t)
\end{equation}
in the inequality
\begin{equation} \label{Eq_1.3ABC}
\max_{|\bs\ell|=1}\left | \frac{\partial \bs u}{ \partial \bs\ell }(x, t) \right |\leq {\mathcal K}_p(t) \nl \bs\varphi \nr_{p}
\end{equation}
is found. 
In particular, 
\begin{equation} \label{Eq_1.4}
{\mathcal K}_\infty(t)=\frac{1}{\sqrt{\pi }}\bv \;{\mathcal I}_A^{-1/2}(t) \bv  
\bv \;e^{{\mathcal I}_{C^*}(t)}\bv \;.
\end{equation}
For the single parabolic equation (\ref{Eq_SHE}) with $f=0$, the previous 
formula takes the form
\begin{equation} \label{Eq_1.4SPE}
{\mathcal K}_\infty(t)=\frac{1}{\sqrt{\pi }}\bv \;{\mathcal I}_A^{-1/2}(t) \bv  \exp\left \{ \int_0^t c(t)dt \right \} \;.
\end{equation}
As a special case of (\ref{Eq_1.4SPE}) one has
$$
{\mathcal K}_\infty(t)=\frac{1}{\sqrt{\pi }}\;\frac{\exp\left \{ \int_0^t c(t)dt \right \}} 
{\left \{\int_0^t a(t)dt \right \}^{1/2}}
$$
for $A(t)=a(t)I$, where $I$ is the unit matrix of order $n$. 

In Section \ref{S_5} we consider a solution of the Cauchy problem 
\begin{eqnarray}  \label{NH}
\left\{\begin{array}{ll}
\displaystyle{\frac{\partial \bs u}{\partial t}= 
\sum_{j,k=1}^n a_{jk}(t)\frac{\partial^2 \bs u}{\partial x_j \partial x_k}}+\sum_{j=1}^n b_{j}(t)\frac{\partial \bs u}{\partial x_j}
+C(t) \bs u+\bs f(x,t)& \quad{\rm in}\; {\mathbb R}^{n+1}_T, \\
       \\
\displaystyle{ \bs u\big |_{t=0}=\bs 0}\;,  
\end{array}\right . 
\end{eqnarray}
where $\bs f\in [L^p({\mathbb R}^{n+1}_T)]^m\cap [C^\alpha\big (\overline{{\mathbb R}^{n+1}_T} \big )]^m $, $\alpha \in (0, 1)$. 
By $[C^\alpha \big (\overline{{\mathbb R}^{n+1}_T} \big )]^m$ we denote the space of $m$-component vector-valued functions $\bs f(x, t)$ which are continuous and bounded in $\overline{{\mathbb R}^{n+1}_T}$ and locally H\"older continuous with exponent $\alpha$ in $x\in {\mathbb R} ^n$, uniformly with respect to $t\in[0, T]$.
The space $[L^p\big ({\mathbb R}^{n+1}_T\big )]^m$ is endowed with the norm
$$
\nl \bs f \nr_{p, T}=\left \{ \int_0^T\int _ {{\mathbb R}^n } \big | (\bs f(x, \tau ) \big |^p dx d\tau\right \}^{1/p}
$$
for $1\leq  p< \infty $, and 
$$
\nl \bs f \nr_{\infty,T} =\mbox{ess}\;\sup \{ | \bs f(x, \tau) |:
x \in {{\mathbb R}^n },\; \tau \in (0, T) \}.
$$

In this section we also obtain two groups of sharp estimates. 
As before, first of them concerns modulus to solution $\bs u$ of problem (\ref{NH}). Namely, we prove the inequality 
\begin{equation} \label{Eq_3.007N}
|\bs u(x, t) |\leq {\mathcal N}_p(t) \nl \bs f \nr_{p, t}
\end{equation}
with the sharp coefficient 
\begin{equation} \label{Eq_3.007AN}
{\mathcal N}_p(t)=\frac{1}{(2\sqrt{\pi })^{n/p}p'^{\;n/(2p')}}
\max_{|\bs z|=1}
\left \{\int_0^t \frac{
\left | e^{{\mathcal I}_{C^*}(t, \tau)}\bs z\right |^{p'}}{\big (\det {\mathcal I}_A^{1/2}(t, \tau)\big )^{p'-1}}\;d \tau \!\right \}^{1/p'}\;,
\end{equation}
where $(x, t)$ is an arbitrary point in the layer ${\mathbb R}^{n+1}_T$. Here and henceforth
${\mathcal I}_F(t,\tau)=\int_\tau^t F(t)dt $, where $F$ can be a vector-valued or matrix-valued function. Formula (\ref{Eq_3.007AN}) is obtained under assumption that the integral in (\ref{Eq_3.007AN}) converges. For instance, integral in (\ref{Eq_3.007AN}) is convergent for $p>(n+2)/2$ in the case of single parabolic equation with constant coefficients.  

As a special case of (\ref{Eq_3.007AN}) one has
\begin{equation} \label{Eq_3.007BN}
{\mathcal N}_\infty (t)= \max_{|\bs z|=1}\;\int_0^t \left | e^{{\mathcal I}_{C^*}(t, \tau)}\bs z\right |d\tau.
\end{equation}
For the single parabolic equation (\ref{Eq_SHE}) the previous 
formula takes the form
$$
{\mathcal N}_\infty(t)=\int_0^t\exp\left \{ \int_\tau ^t c(t)dt \right \}d\tau \;.
$$

The second group of sharp estimates concerns 
$\partial \bs u/ \partial \bs\ell $.  
The explicit formula for the sharp coefficient
\begin{equation} \label{Eq_1.5N}
{\mathcal C}_{p, \bs\ell}(t)\!=
\!\frac{1}{\big \{ 2^{n}\pi^{(n\!+\!p\!-\!1)/2} \big \}^{1/p}}\left \{\!\! \frac{\Gamma \left (\frac{p'+1}{2} \right )} {p'^{\;(n+p')/2}} \!\!\right \}^{\!\!1/p'}\!\!\!\!\!\!\!\max_{|\bs z|=1}
 \left \{ \int_0^t\!\frac{\left | {\mathcal I}_A^{-1/2}(t,\tau)\bs\ell\right |^{p'}\!\!\left | e^{{\mathcal I}_{C^*}(t, \tau)}\bs z\right |^{p'}} {\left (\det {\mathcal I}_A^{1/2}(t,\tau) \right )^{p'-1}} d\tau
 \right \}^{\!1/p'}
\end{equation} 
in the inequality
\begin{equation} \label{Eq_1.5A} 
\left | \frac{\partial \bs u}{ \partial \bs\ell }(x, t) \right |\leq {\mathcal C}_{p, \bs\ell}(t)\nl \bs f \nr_{p, t}
\end{equation} 
for solutions of the Cauchy problem (\ref{NH}) is found, where $(x, t)\in {\mathbb R}^{n+1}_T$.

Formula (\ref{Eq_1.5N}) is obtained under assumption that the integral in (\ref{Eq_1.5N}) converges.
We note that this integral is convergent for $p>n+2$ in the case of single parabolic equation with constant coefficients (see \cite{KM4}).  
  
As a consequence of (\ref{Eq_1.5N}), we arrive at the formula for the sharp coefficient
\begin{equation} \label{Eq_1.5AB}
{\mathcal C}_p(t)=\max_{|\bs\ell |=1}{\mathcal C}_{p, \bs\ell}(t)
\end{equation} 
in the inequality
\begin{equation} \label{Eq_1.5ABC}
\max_{|\bs\ell|=1}\left | \frac{\partial \bs u}{ \partial \bs\ell }(x, t) \right |\leq {\mathcal C}_p(t) \nl \bs f \nr_{p, t}\;.
\end{equation} 
For instance, 
\begin{equation} \label{Eq_1.6}
{\mathcal C}_\infty(t)=\frac{1}{\sqrt{\pi}}\;\max_{|\bs\ell|=1}
\max_{|\bs z|=1}\int_0^t
\left | {\mathcal I}_A^{-1/2}(t,\tau)\bs\ell\right |\left | e^{{\mathcal I}_{C^*}(t, \tau)}\bs z\right |\; d\tau\;.
\end{equation}
For the single parabolic equation (\ref{Eq_SHE}) the previous 
formula takes the form
\begin{equation} \label{Eq_1.6A}
{\mathcal C}_\infty(t)=\frac{1}{\sqrt{\pi}}\;\max_{|\bs\ell|=1}
\int_0^t\left | {\mathcal I}_A^{-1/2}(t,\tau)\bs\ell\right | 
\exp \left \{ \int_\tau ^t c(t)dt \right \}\; d\tau\;.
\end{equation}
In the particular case $A(t)=a(t)I$, formula (\ref{Eq_1.6A}) becomes
$$
{\mathcal C}_\infty(t)=\frac{1}{\sqrt{\pi}}\;\int_0^t\;
\frac{\exp\left \{ \int_\tau^t c(t)dt \right \}} 
{\left \{\int_\tau^t a(t)dt \right \}^{1/2}}\; d\tau\;.
$$

Note that the sharp coefficients  (\ref{Eq_1.3P}) and (\ref{Eq_1.5N}) do not depend on the coefficient  vector $b(t)=(b_1(t), \dots, b_n(t))$. 

\section{The norm of a certain integral operator} \label{S_2}

Let $({\cal X}, {\cal A}, \mu )$ be a measure space and let 
$1\leq p\leq \infty $. We introduce the space $[L^p({\mathcal X}, {\mathcal A}, \mu )]^n$ of real vector-valued functions endowed with the norm
\begin{equation} \label{EGLP_2.01}
\nl\bs f\nr_p=\left \{ \int _{\mathcal X} |\bs f(x)|^p d\mu (x) \right \}^{1/p}
\end{equation}
for $1\leq  p< \infty $, and 
$$
\nl\bs f\nr_\infty =\mbox{ess}\;\sup \{ |\bs f(x) |: x \in {\mathcal X} \}\;.
$$
By $(\eta,  \zeta )$ we denote the inner product of the vectors $\eta $ and $\zeta $ in a Euclidean space.

The following assertion was proved in \cite{KM1} (Proposition 1.2). It contains a representation of the norm $\nl S \nr_p$ of the integral operator $S$ defined on $[L^p({\mathcal X}, {\mathcal A}, \mu )]^n$ and acting into ${\mathbb R}^m$. 

\begin{proposition} \label{P_1.2}
Let $G=((g_{ij}))$ be an $(m\times n)$-matrix-valued function with the elements
$g_{ij} \in L^{p'}({\mathcal X}, {\mathcal A}, \mu )$ 
whose values $g_{ij}$ are everywhere finite.
The norm of the linear continuous operator $S: [L^p({\mathcal X}, {\mathcal A}, \mu )]^n
\rightarrow {\mathbb R}^m$ defined by
\begin{equation} \label{EGLP_2.07}
S(\bs f)=\int _{\mathcal X} G(x) \bs f(x) d\mu (x)
\end{equation}
is equal to
\begin{equation} \label{EGLP_2.08}
\nl S\nr_p =\sup _{|\bs z|=1} \nl G^* \bs z \nr_{p'},
\end{equation}
where $G^*$ stands for the transposed matrix of $G$,  $\bs z \in {\mathbb R}^m$ and $p'$ is defined by $1/p+1/p'=1$.
\end{proposition}
\begin{proof} {\it 1. Upper estimate for $\nl S \nr_p$}. For any vector $\bs z \in {\mathbb R}^m$,
\begin{equation} \label{EGLP_2.10}
(S(\bs f), \bs z)=\int _{\mathcal X} (G(x) \bs f(x), \bs z) d\mu (x) = \int _{\mathcal X} (\bs f(x), G^*(x)\bs z ) d\mu (x).
\end{equation}
Hence by H\"older's inequality
$$
|(S(\bs f), \bs z)| \leq \int _{\mathcal X} |(\bs f(x), G^*(x)\bs z )| d\mu (x)
\leq \int _{\mathcal X} |G^*(x)\bs z| | \bs f(x)| d\mu (x)\leq \nl G^* \bs z \nr_{p'} \nl\bs f \nr_p.
$$
Therefore, taking into account that $|S(\bs f)|=\sup \{|(S(\bs f), \bs z)| : 
|\bs z|=1 \}$ we arrive at the estimate
\begin{equation} \label{EGLP_2.11}
\nl S\nr_p \leq \sup _{|\bs z|=1} \nl G^* \bs z\nr_{p'}.
\end{equation}

{\it 2. Lower estimate for $\nl S \nr_p$}. Let us fix 
$\bs z \in  {\mathbb S}^{m-1}=\{ \bs z \in {\mathbb R}^m : |\bs z|=1 \}$.
We introduce the vector-valued function with $n$ components
\begin{equation} \label{EGLP_2.12}
\bs h_{\bs z}(x)= \bs g_{\bs z}(x) h(x),
\end{equation}
where $h \in L_p({\mathcal X}, {\mathcal A}, \mu ),\; \nl h \nr_p \leq 1$, and
\begin{eqnarray*} 
\bs g_{\bs z}(x) =\left \{
\begin{array}{lll}
   G^*(x) \bs z |G^*(x) \bs z|^{-1} & \quad\hbox {for}\quad |G^*(x) \bs z|\not= 0 ,\\
     & \\
  {\bs 0} & \quad\hbox {for}\quad |G^*(x) \bs z|= 0 .
\end{array}\right .
\end{eqnarray*}
Note that $\bs h_{\bs z} \in [L^p({\cal X}, {\cal A}, \mu )]^n$ and  $\nl\bs h_{\bs z}\nr_p \leq 1$.
Setting (\ref{EGLP_2.12}) as $\bs f$ in (\ref{EGLP_2.10}) we find
\[
(S(\bs h_{\bs z}), \bs z)=(S(\bs g_{\bs z}h), \bs z)=\int _{\cal X} (\bs g_{\bs z}(x), G^*(x)\bs z)h(x) d\mu (x)=
\int _{\cal X} |G^*(x)\bs z|h(x) d\mu (x).
\]
Hence
\begin{eqnarray*}
\nl S\nr_p&=&\sup _{\nl\bs f \nr_p\leq 1}|S(\bs f)|\geq 
\sup _{\nl h\nr_p\leq 1}|S(\bs g_{\bs z}h)|\geq
\sup _{\nl h \nr_p\leq 1}|(S(\bs g_{\bs z}h), \bs z)|\\
& &\\
&=&\sup _{\nl h\nr_p\leq 1}\left | \int _{\mathcal X} |G^*(x)\bs z|h(x) d\mu (x) \right |
=\nl G^* \bs z \nr_{p'}.
\end{eqnarray*}
By the arbitrariness  of $\bs z \in {\mathbb S}^{m-1}$,
\begin{equation} \label{EGLP_2.13}
\nl S \nr_p \geq  \sup _{|\bs z|=1} \nl G^* \bs z \nr_{p'},
\end{equation}
which together with (\ref{EGLP_2.11}) leads to (\ref{EGLP_2.08}).
\end{proof}

\begin{remark} {\bf 1}. \label{r_alter} Let $1<p\leq \infty$. Estimate
(\ref{EGLP_2.13}) can be derived with the help of the function
\begin{eqnarray} \label{alter}
\bs h_{\bs z}(x) =\left \{
\begin{array}{lll}\displaystyle{
   \frac{G^*(x) \bs z |G^*(x) \bs z|^{p'-2}}{\nl G^* z \nr_{p'}^{p'/p}}} 
   & \quad\hbox {for}\quad |G^*(x) \bs z|\not= 0 ,\\
     & \\
  {\bs 0} & \quad\hbox {for}\quad |G^*(x) \bs z|= 0 ,
\end{array}\right .
\end{eqnarray}
where $\bs z \in {\mathbb S}^{m-1}$. Indeed, since $p'/(p'-1)=p$, it follows that
\[
\nl\bs h_z\nr_p=\nl G^*z\nr_{p'}^{-p'/p}\left \{  \int _{\cal X} |G^*(x)\bs z|^{(p'-1)p} 
d\mu (x)\right \}^{1/p}=\nl G^*z \nr_{p'}^{-p'/p}\nl G^*z\nr_{p'}^{p'/p}=1.
\]
Using (\ref{EGLP_2.10}) and (\ref{alter}), we obtain
\begin{eqnarray*}
\nl S\nr_p&\!\!\!=\!\!\!&\sup _{\nl \bs f \nr_p\leq 1}|S(\bs f)|\geq (S(\bs h_{\bs z}), \bs z)
=\nl G^*z \nl_{p'}^{-p'/p}\int _{\mathcal X} |G^*(x)\bs z|^{p'} d\mu (x) \\
& &\\
&\!\!\!=\!\!\!& \nl G^*z \nr_{p'}^{-p'/p}\nl G^* \bs z\nr_{p'}^{p'}=\nl G^* \bs z\nr_{p'},
\end{eqnarray*}
which implies (\ref{EGLP_2.13}), because $\bs z \in {\mathbb S}^{m-1}$ is arbitrary.  
\end{remark}

\section{Formulas for solutions of weakly coupled parabolic systems} \label{S_3}
By $(x,  y )$ we mean the inner product of the vectors $x $ and $y $ in ${\mathbb R}^n$.
The Schwartz class of $m$-component vector-valued rapidly decreasing $C^{\infty}$-functions on ${\mathbb R}^n$ will be denoted by $[{\mathcal S}(\mathbb R^n)]^m$.

Since the matrix $A(t)$ is positive definite for any $t\in [0, T]$, the matrix ${\mathcal I}_A(t)=\int_0^t A(t)dt$ is positive definite for any $t\in (0, T]$. So, there exist positive definite matrices ${\mathcal I}_A^{1/2}(t)$ and ${\mathcal I}_A^{-1/2}(t)$ of order $n$ such 
that $\big ( {\mathcal I}_A^{1/2}(t) \big )^2={\mathcal I}_A(t)$ and $\big ( {\mathcal I}_A^{-1/2}(t) \big )^2={\mathcal I}_A^{-1}(t)$ for any $t\in (0, T]$(e.g. \cite{LA}, sect. 2.14).

We start with the Cauchy problem (\ref{H}) for the homogeneous equation. 
The assertion below can be proved analogously to the corresponding statement for the heat equation (e.g.\cite{Hu}, Th. 5.4; \cite{OO}, Sect. 4.8; \cite{MS},  Sect. 7.4). We restrict 
ourselves to a formal argument.
 
\begin{lemma} \label{L_01} Let $\bs\varphi \in [{\mathcal S}(\mathbb R^n)]^m$. A solution $\bs u$ of problem $(\ref{H})$ is given by
\begin{equation} \label{Eq_3.02}
\bs u(x, t)=\int_{{\mathbb R}^n}G(x-y, t)\bs\varphi(y)dy\;,
\end{equation}
where
\begin{equation} \label{Eq_3.03}
G(x, t)= \frac{e^{{\mathcal I}_C(t)}}{(2\sqrt{\pi })^n\det {\mathcal I}_A^{1/2}(t)} e^ {-\left |{\mathcal I}_A^{-1/2}(t)( x+{\mathcal I}_b(t))\right |^2\big /4}\;.
\end{equation}
\end{lemma} 
\begin{proof} Let $\bs u$ be solution of the Cauchy problem (\ref{H}). We introduce the function 
\begin{equation} \label{Eq_3.02B}
\bs v= e^{-{\mathcal I}_C(t)}\bs u.
\end{equation}
Then $\bs v$ is solution of the problem
\begin{eqnarray} \label{Hv}
\left\{\begin{array}{ll}
\displaystyle{\frac{\partial \bs v}{\partial t}= \sum_{j,k=1}^n a_{jk}(t)\frac{\partial^2 \bs v}{\partial x_j \partial x_k}}+\sum_{j=1}^n b_{j}(t)\frac{\partial \bs v}{\partial x_j}
& \quad{\rm in}\; {\mathbb R}^{n+1}_T, \\
       \\
\displaystyle{\bs  v\big |_{t=0}=\bs \varphi }\;,  
\end{array}\right .
\end{eqnarray}

Applying the Fourier transform
\begin{equation} \label{eq_4.02}
\hat{\bs v}(\xi, t)= \frac{1}{(2\pi)^{n/2}} \int_{\mathbb R ^n} e ^{-i( x,\xi)}\bs v(x, t)dx
\end{equation}
to the Cauchy problem (\ref{Hv}), we obtain
\begin{equation} \label{Eq_4.03}
\frac{d\hat{\bs v}}{dt}=\big \{-(A(t)\xi, \xi)+i (b(t), \xi) \big \}\hat{\bs v}\;,\;\;\;\;\;\hat{\bs v}(\xi,0)=\hat{\bs\varphi}(\xi)\;.
\end{equation}
The solution of problem (\ref{Eq_4.03}) is
\begin{equation} \label{Eq_4.04}
\hat{\bs v}(\xi, t)=\hat{\bs\varphi }(\xi)e^{\int_0^t \{-(A(t)\xi, \xi )+i (b(t), \xi) \}dt}=
\hat{\bs\varphi }(\xi)e^{ -({\mathcal I}_A(t)\xi, \xi )+i ({\mathcal I}_b(t), \xi) }.
\end{equation}

By the inverse Fourier transform
$$
\bs v(x, t)=\frac{1}{(2\pi)^{n/2}}\int_{\mathbb R ^n} e ^{i( x, \xi)}\hat{\bs v}(\xi ,t)d\xi, 
$$
we deduce from (\ref{Eq_4.04})
\begin{eqnarray} \label{Eq_4.05a}
& & \bs v(x, t)= \frac{1}{(2\pi)^{n/2}}\int_{\mathbb R ^n} e ^{i( x,\xi)}\hat{\bs\varphi}(\xi)e^{-({\mathcal I}_A(t)\xi, \xi)+i ({\mathcal I}_b(t), \xi) }d\xi \nonumber \\
& & = \frac{1}{(2\pi)^{n}}\int_{\mathbb R ^n} e ^{i( x,\xi)}e^{ -({\mathcal I}_A(t)\xi, \xi )+i ({\mathcal I}_b(t), \xi) }  \left \{ \int_{\mathbb R ^n} e ^{-i( y, \xi)}{\bs\varphi}(y)dy \right \} d\xi \nonumber\\
& & =\frac{1}{(2\pi)^{n}}\int_{\mathbb R ^n} \left\{ \int_{\mathbb R ^n} e^{i( x- y+ {\mathcal I}_b(t), \xi)}e^{-({\mathcal I}_A(t) \xi, \xi)} d\xi \right\} {\bs\varphi}(y)dy\;.  
\end{eqnarray}
Let us denote
\begin{equation} \label{Eq_4.06}
    G_0(x,t)=  \frac{1}{(2\pi)^n}\int_{\mathbb R ^n} e^{-({\mathcal I}_A(t)\xi, \xi)+i( x+ {\mathcal I}_b(t) , \xi)}d\xi\;.
\end{equation}

The known formula (e.g.\cite{VL}, Ch. 2, Sect. 9.7)
\begin{equation} \label{Eq_4.06A}
\int_{\mathbb R ^n} e^{-(M\xi, \xi)+i( \zeta,\xi)}d\xi=
\frac{\pi^{n/2}}{\sqrt{\det M}}e^{-(M^{-1}\zeta, \zeta )/4},
\end{equation}
where $M$ is a symmetric positive definite matrix of order $n$,
together with (\ref{Eq_4.06}) leads to
\begin{equation} \label{Eq_4.07}
   G_0(x,t)=  \frac{1}{(2\sqrt{\pi })^n\sqrt{\det {\mathcal I}_A(t)}}e^{-\left ({\mathcal I}_A^{-1}(t)( x+{\mathcal I}_b(t)) ,  x+{\mathcal I}_b(t) \right )\big /4}\;.
\end{equation}

Since $(B^{-1}\zeta, \zeta)=(B^{-1/2}B^{-1/2}\zeta, \zeta)=
(B^{-1/2}\zeta, B^{-1/2}\zeta )=|B^{-1/2}\zeta|^2$ as well as $\det B=\big ( \det B^{1/2}\big )^2$ for any symmetric positive
definite matrix $B$ and every vector  $\zeta \in {\mathbb R ^n}$, we can write (\ref{Eq_4.07}) as 
\begin{equation} \label{Eq_4.08}
G_0(x, t)= \frac{1}{(2\sqrt{\pi })^n\det {\mathcal I}_A^{1/2}(t)} e^ {-\left |{\mathcal I}_A^{-1/2}(t)( x+{\mathcal I}_b(t))\right |^2\big /4}\;.
\end{equation}
It follows from (\ref{Eq_3.02B}), (\ref{Eq_4.05a}), (\ref{Eq_4.06}) and (\ref{Eq_4.08}) that the solution of problem (\ref{H}) can be represented as (\ref{Eq_3.02}),
where $G(x,t)$ is given by (\ref{Eq_3.03}).
\end{proof}

\medskip
The next assertion can be proved in view of (\ref{Eq_3.02B}) and (\ref{Hv}) on the base of Lemma \ref{L_01}
similarly to the analogous statement for the heat equation (e.g.\cite{Hu}, Th. 5.5; \cite{MS},  Sect. 7.4). 

\begin{proposition} \label{P_1} Suppose that $1\leq p\leq \infty$
and $\bs\varphi \in [L^p({\mathbb R}^n)]^m$. Define
$\bs u: {\mathbb R}^{n+1}_T\rightarrow {\mathbb R}^m$
by $(\ref{Eq_3.02})$, where $G$ is given by $(\ref{Eq_3.03})$.
Then $\bs u(x, t)$ is solution of the system
$$
\frac{\partial \bs u}{\partial t}= \sum_{j,k=1}^n a_{jk}(t)\frac{\partial^2 \bs u}{\partial x_j \partial x_k}+
\sum_{j=1}^n b_{j}(t)\frac{\partial \bs u}{\partial x_j}+
C(t) \bs u
$$ 
in ${\mathbb R}^{n+1}_T$. If $1\leq p< \infty$, then 
$\bs u(\cdot , t)\rightarrow \bs\varphi$ in $[L^p]^m$ as $t \rightarrow 0^+$.
\end{proposition}

\begin{remark} {\bf 2.}
It is known (e.g. \cite{FR}, Ch.1, Sect. 6, 7 and 9) that
$$
\bs v(x, t)=\int_{{\mathbb R}^n}G_0(x-y, t)\bs\varphi(y)dy\;,
$$
where $G_0$ is given by (\ref{Eq_4.07}), represents a unique bounded solution of the Cauchy problem (\ref{Hv}) and $\bs v(x , t)\rightarrow \bs\varphi(x)$ as $t \rightarrow 0^+$ at any $x\in {\mathbb R}^n$ under assumption that $\bs\varphi$ belongs to $[C({\mathbb R}^n)]^m\cap [L^\infty ({\mathbb R}^n)]^m$. This fact together with (\ref{Eq_3.02B}) and the Lusin's theorem 
(see \cite{VU}, Ch. VI, Sect. 6) implies that $\bs u(\cdot , t)\rightarrow \bs\varphi $ almost everywhere in ${\mathbb R}^n$ as $t \rightarrow 0^+$, where $\bs u$ is a solution of problem (\ref{H}) with $\bs\varphi \in [L^\infty({\mathbb R}^n)]^m$.
\end{remark}

\medskip
Further, let us consider solution $\bs u$ of the Cauchy problem (\ref{NH}) for the nonhomogeneous system. We introduce the function (\ref{Eq_3.02B}). Then $\bs v$ is solution of the problem
\begin{eqnarray} \label{NvNH}
\left\{\begin{array}{ll}
\displaystyle{\frac{\partial \bs v}{\partial t}= \sum_{j,k=1}^n a_{jk}(t)\frac{\partial^2 \bs v}{\partial x_j \partial x_k}}+\sum_{j=1}^n b_{j}(t)\frac{\partial \bs v}{\partial x_j}+e^{-{\mathcal I}_C(t)}\bs f(x, t)
& \quad{\rm in}\; {\mathbb R}^{n+1}_T, \\
       \\
\displaystyle{\bs  v\big |_{t=0}=\bs 0 }\;.  
\end{array}\right .
\end{eqnarray}
In view of (\ref{Eq_3.02B}) and (\ref{NvNH}), the next statement can be proved analogously to the similar assertion for the heat equation (e.g. \cite{OO}, Sect. 4.8). 
We restrict ourselves to a formal argument.

\begin{lemma} \label{L_2} Let $\bs f(\cdot, t)\in [{\mathcal S}({\mathbb R}^n)]^m$ for any $t\in [0, T]$ and let 
the quantities $C_{\alpha, m}$ in the estimates
$$
\left (1+|x|^m \right )|\partial_x^\alpha \bs f(x, t)|\leq C_{\alpha, m}
$$
are independent of $t$ for any integer $m\geq 0$ and multiindex $\alpha$.

The solution of problem $(\ref{NH})$  is given by
\begin{equation} \label{Eq_3.2a}
\bs u(x, t)=\int_0^t\int_{{\mathbb R}^n}P(x-y, t,\tau) \bs f(y, \tau)dy d\tau\;,
\end{equation}
where $P$ is defined by 
\begin{equation} \label{Eq_4.6B}
P(x,t, \tau )=\frac{e^{{\mathcal I}_C(t,\tau)}}{(2\sqrt{\pi })^n\det {\mathcal I}_A^{1/2}(t,\tau)} e^ {-\left |{\mathcal I}_A^{-1/2}(t,\tau)( x+{\mathcal I}_b(t,\tau))\right |^2\big /4}.
\end{equation}
\end{lemma}
\begin{proof} 
Applying the Fourier transform to the Cauchy problem (\ref{NvNH}), we obtain
\begin{equation} \label{Eq_4.03A}
\frac{d\hat{\bs v}}{dt}=\big \{-(A(t)\xi, \xi)+i (b(t), \xi) \big \}\hat{\bs v}+e^{-{\mathcal I}_C(t)}
\hat{\bs f}(\xi, t)\;,\;\;\;\;\;\hat{\bs v}(\xi,0)=0\;.
\end{equation}
The solution of problem (\ref{Eq_4.03A}) is
\begin{eqnarray} \label{Eq_4.4A}
\hat{\bs v}(\xi, t)&=&e^{ -({\mathcal I}_A(t)\xi, \xi)+i ({\mathcal I}_b(t), \xi)}
\int_0^t e^{-{\mathcal I}_C(\tau)}\hat{\bs f}(\xi, \tau)e^{ ({\mathcal I}_A(\tau)\xi, \xi)-i ({\mathcal I}_b(\tau), \xi)}d\tau\nonumber\\
&=&\int_0^t e^{-{\mathcal I}_C(\tau)}\hat{\bs f}(\xi, \tau)
e^{\int_\tau ^t\{ -(A(s)\xi, \xi)+i (b(s), \xi)\}ds} d\tau\;.
\end{eqnarray}
By the inverse Fourier transform in (\ref{Eq_4.4A}), we have
\begin{eqnarray*} 
\bs v(x, t)&=&\frac{1}{(2\pi)^{n/2}}\int_{\mathbb R ^n}\left \{\int_0^t e^{-{\mathcal I}_C(\tau)}\hat{\bs f}(\xi, \tau)e^{\int_\tau ^t\{ -(A(s)\xi, \xi)+i (b(s), \xi)\}ds}d\tau \right \}e^{i(x, \xi)}d\xi\\
&=&\frac{1}{(2\pi)^{n/2}}\int_0^t\left \{\int_{\mathbb R ^n} e^{i(x, \xi)} e^{-{\mathcal I}_C(\tau)}\hat{\bs f}(\xi, \tau)e^{\int_\tau ^t\{ -(A(s)\xi, \xi)+i (b(s), \xi)\}ds} d\xi \right \}d\tau\\
&=&\frac{1}{(2\pi)^n}\int_0^te^{-{\mathcal I}_C(\tau)}\left \{\int_{\mathbb R ^n} 
e^{i(x, \xi)} 
\left \{ \int_{\mathbb R ^n}e^{-i(y, \xi)}\bs f(y, \tau)dy \right \}
e^{\int_\tau ^t\{ -(A(s)\xi, \xi)+i (b(s), \xi)\}ds} d\xi \right \}d\tau
\\
&=&\int_0^t \int_{\mathbb R ^n}
\left \{\frac{1}{(2\pi)^n} \int_{\mathbb R ^n}e^{i(x-y+\int_\tau ^t b(s)ds, \xi)}
e^{ -\int_\tau ^t(A(s) \xi, \xi)ds}d\xi\right \}e^{-{\mathcal I}_C(\tau)}\bs f(y, \tau)dy d\tau.
\end{eqnarray*}
Multiplying the last equality by 
$$
e^{{\mathcal I}_C(t)},
$$
in view of (\ref{Eq_3.02B}), we arrive at
\begin{equation} \label{Eq_4.5A}
\bs u(x, t)=\int_0^t \int_{\mathbb R ^n}
\left \{\frac{e^{{\mathcal I}_C(t, \tau)}}{(2\pi)^n} \int_{\mathbb R ^n}e^{i(x-y+\int_\tau ^t b(s)ds, \xi)}
e^{ -\int_\tau ^t(A(s) \xi, \xi)ds}d\xi\right \}\bs f(y, \tau)dy d\tau.
\end{equation}
By (\ref{Eq_4.06A}), expression inside of the braces in the right-hand side of (\ref{Eq_4.5A}) is equal to
$$
\frac{e^{{\mathcal I}_C(t,\tau)}}{(2\sqrt{\pi })^n\sqrt{\det {\mathcal I}_A(t,\tau)}} e^ {-\left ({\mathcal I}_A^{-1}(t,\tau)( x-y+{\mathcal I}_b(t,\tau)), x-y+{\mathcal I}_b(t,\tau)\right )/4}.
$$
Applying the same arguments as in the proof of Lemma \ref{L_01},
we can rewrite (\ref{Eq_4.5A}) as (\ref{Eq_3.2a}), where $P(x, t, \tau)$
is given by (\ref{Eq_4.6B}).
\end{proof}

\section{Sharp estimates for solutions to the homogeneous weakly coupled parabolic system} \label{S_4}

In this section we obtain estimates for solution of the Cauchy problem (\ref{H}). First, we prove the sharp pointwise estimate for $|\bs u|$ with  $ \bs \varphi \in [L^p({\mathbb R}^n)]^m$, where $p \in [1, \infty]$.

\setcounter{theorem}{0}
\begin{theorem} \label{T_01} Let $(x, t)$ be an arbitrary point in ${\mathbb R}^{n+1}_T$ and $\bs u $ be solution of problem $(\ref{H})$. 
The sharp coefficient ${\mathcal H}_p(t)$ in inequality $(\ref{Eq_3.007})$ is given by $(\ref{Eq_3.007A})$.
As a special case of $(\ref{Eq_3.007A})$ with $p=\infty$ one has $(\ref{Eq_3.007B})$.
\end{theorem}
\begin{proof} By Proposition \ref{P_1.2} and (\ref{Eq_3.02}),
(\ref{Eq_3.03}), the sharp coefficient in inequality (\ref{Eq_3.007}) is given by
\begin{equation} \label{Eq_3.008}
{\mathcal H}_p(t)=\max_{|\bs z|=1}
\frac{\left |\left (e^{{\mathcal I}_C(t)}\right )^*\bs z\right |}{(2\sqrt{\pi })^n\det {\mathcal I}_A^{1/2}(t)}
\left \{\int_{\mathbb R ^n}e^ {-p'\left |{\mathcal I}_A^{-1/2}(t)( x-y+{\mathcal I}_b(t))\right |^2\big /4} dy\right \}^{1/p'}.
\end{equation}
Since
\begin{equation} \label{Eq_3.002A}
\left ( e^{{\mathcal I}_C(t)} \right )^*= e^{{\mathcal I}_{C^*}(t)},
\end{equation}
it follows from  (\ref{Eq_3.008}) that
\begin{equation} \label{Eq_3.009A}
{\mathcal H}_p(t)=
\frac{\bv e^{{\mathcal I}_{C^*}(t)}\bv }{(2\sqrt{\pi })^n\det {\mathcal I}_A^{1/2}(t)}
\left \{\int_{\mathbb R ^n}e^ {-p'\left |{\mathcal I}_A^{-1/2}(t)( x-y+{\mathcal I}_b(t))\right |^2\big /4} dy\right \}^{1/p'}.
\end{equation}
Further, we introduce the new variable
$\xi={\mathcal I}_A^{-1/2}(t)( x-y+{\mathcal I}_b(t) )$.
Since $y=-{\mathcal I}_A^{1/2}(t)\xi +x+{\mathcal I}_b(t)$, we have $dy =\det {\mathcal I}_A^{1/2}(t)d\xi$, 
which together with (\ref{Eq_3.009A}) leads to the following representation 
\begin{equation} \label{Eq_3.010}
{\mathcal H}_p(t)=
\frac{\bv e^{{\mathcal I}_{C^*}(t)}\bv }{(2\sqrt{\pi })^n
\left ( \det {\mathcal I}_A^{1/2}(t)\right )^{1/p}}
\left \{\int_{\mathbb R ^n}e^ {-p'\left |\xi \right |^2 /4} d\xi\right \}^{1/p'}.
\end{equation}
Passing to the spherical coordinates in (\ref{Eq_3.010}), we obtain
\begin{eqnarray} \label{Eq_3.011}
{\mathcal H}_p(t)&=&
\frac{\bv e^{{\mathcal I}_{C^*}(t)}\bv }{(2\sqrt{\pi })^n
\left ( \det {\mathcal I}_A^{1/2}(t)\right )^{1/p}}
\left \{\int_{{\mathbb S}^{n-1}}d\sigma \int_0^{\infty} \rho^{n-1}e^{-p'\rho^2/4} d\rho\right \}^{1/p'}\nonumber\\
&=&\frac{\bv e^{{\mathcal I}_{C^*}(t)}\bv }{(2\sqrt{\pi })^n
\left ( \det {\mathcal I}_A^{1/2}(t)\right )^{1/p}}
\left \{\omega_n \int_0^{\infty} \rho^{n-1}e^{-p'\rho^2/4} d\rho\right \}^{1/p'},
\end{eqnarray}
where $\omega_n=2\pi^{n/2}/\Gamma (n/2)$ is the area of the unit sphere ${\mathbb S}^{n-1}$ in ${\mathbb R}^{n}$.
Further, making the change of variable $\rho=\sqrt{u}$ in the integral 
$$
\int_0^{\infty}\rho^{n-1} e^{-p'\rho^2/4 } d\rho
$$
and using the formula (e.g. \cite{GRJ}, 3.381, item 4) 
\begin{equation} \label{Eq_3.012}
\int_0^\infty x^{\alpha-1}e^{-\beta x}dx=\beta^{-\alpha}\Gamma(\alpha)
\end{equation}
with positive $\alpha$ and $\beta$, we obtain
$$ 
\int_0^{\infty} \rho^{n-1} e^{-p'\rho^2/4 } d\rho=\frac{1}{2}\int_0^{\infty}u^{\frac{n}{2}-1}e^{-p'u/4}du=
\frac{1}{2}\left ( \frac{4}{p'} \right )^{\frac{n}{2}}{\Gamma \left ( \frac{n}{2} \right )}\;,
$$ 
which leads to
\begin{equation} \label{Eq_3.013}
\omega_n\int_0^{\infty} \rho^{n-1} e^{-p'\rho^2/4 } d\rho=
\frac{2\pi^{n/2}}{\Gamma \left ( \frac{n}{2} \right )}\;
\frac{1}{2}\left ( \frac{4}{p'} \right )^{\frac{n}{2}}{\Gamma \left ( \frac{n}{2} \right )}=\frac{2^n \pi^{n/2}}{p'^{\;n/2}}\;.
\end{equation}
Substituting (\ref{Eq_3.013}) into (\ref{Eq_3.011}), we arrive at (\ref{Eq_3.007A}). 
As a particular case of (\ref{Eq_3.007A}) with $p=\infty$, we obtain (\ref{Eq_3.007B}).
\end{proof}

In the next assertion we prove the sharp pointwise estimate for $|\partial \bs u/\partial\bs \ell |$ with  $ \bs \varphi \in [L^p({\mathbb R}^n)]^m$, $p \in [1, \infty]$.
\begin{theorem} \label{T_02} Let $(x, t)$ be an arbitrary point in ${\mathbb R}^{n+1}_T$ and $\bs u $ be solution of problem $(\ref{H})$. The sharp coefficient ${\mathcal K}_{p, \bs\ell}(t)$ in inequality $(\ref{Eq_1.3A})$ is given by $(\ref{Eq_1.3P})$.

As a consequence, the sharp coefficient ${\mathcal K}_p(t)$ in inequality $(\ref{Eq_1.3ABC})$ is given by
\begin{equation} \label{Eq_3.7}
{\mathcal K}_p(t)=\frac{\bv \;{\mathcal I}_A^{-1/2}(t) \bv  \bv \;e^{{\mathcal I}_{C^*}(t)}\bv }
{\big \{ 2^{n}\pi^{(n+p-1)/2}\det{\mathcal I}_A^{1/2}(t) \big \}^{1/p}}
\left \{ \frac{\Gamma \left (\frac{p'+1}{2} \right )} {p'^{(n+p')/2}} \right \}^{1/p'}.
\end{equation}

As a special case of $(\ref{Eq_3.7})$ with $p=\infty$ one has $(\ref{Eq_1.4})$.
\end{theorem}
\begin{proof} Differentiating in (\ref{Eq_3.03}) with respect to $x_j$, $j=1,\dots , n$, we obtain
\begin{equation} \label{Eq_3.8A}
\frac{\partial}{\partial x_j}G(x, t)= -
\frac{e^{{\mathcal I}_C(t)}}{2(2\sqrt{\pi })^n\det {\mathcal I}_A^{1/2}(t)}
\left \{{\mathcal I}_A^{-1}(t)( x+{\mathcal I}_b(t))\right \}_j  
e^ {-\left |{\mathcal I}_A^{-1/2}(t)( x+{\mathcal I}_b(t))\right |^2\big /4},
\end{equation}
which together with (\ref{Eq_3.1AAB}) and (\ref{Eq_3.02}), leads to
\begin{eqnarray} \label{Eq_3.9}
& &\frac{\partial \bs u}{ \partial { \bs\ell}}=(\bs\ell, \nabla_x)\bs u=\sum_{j=1}^n\int_{{\mathbb R}^n}\ell_j\frac{\partial}{\partial x_j}G(x-y, t)\bs\varphi(y)dy=-
\frac{1}{2(2\sqrt{\pi })^n\det {\mathcal I}_A^{1/2}(t)}\nonumber\\
& &\times\int_{{\mathbb R}^n}\!\!\left ({\mathcal I}_A^{-1}(t)( x-y+{\mathcal I}_b(t)), \bs\ell \right )
e^ {-\left |{\mathcal I}_A^{-1/2}(t)( x-y+{\mathcal I}_b(t))\right |^2\big /4}e^{{\mathcal I}_C(t)}\bs\varphi(y)dy .
\end{eqnarray}
Applying Proposition \ref{P_1.2} to (\ref{Eq_3.9}), we conclude that the sharp coefficient in  estimate (\ref{Eq_1.3A})
is given by
\begin{eqnarray*} 
{\mathcal K}_{p, \bs\ell}(t)&=&\max_{|\bs z|=1}\frac{\left |\left (e^{{\mathcal I}_C(t)}\right )^*\bs z\right |}{2(2\sqrt{\pi })^n\det {\mathcal I}_A^{1/2}(t)}\\
&\times&
\left \{\int_{{\mathbb R}^n}\left |\left ({\mathcal I}_A^{-1}(t)( x-y+{\mathcal I}_b(t)) , \bs\ell \right )\right |^{p'}
e^ {-p'\left |{\mathcal I}_A^{-1/2}(t)( x-y+{\mathcal I}_b(t))\right |^2\big /4}dy \right \}^{1/p'},
\end{eqnarray*}
which, in view of (\ref{Eq_3.002A}), implies
\begin{eqnarray*} 
{\mathcal K}_{p, \bs\ell}(t)&=&\frac{\bv e^{{\mathcal I}_{C^*}(t)}\bv }{2(2\sqrt{\pi })^n\det {\mathcal I}_A^{1/2}(t)}\\
&\times&\left \{\int_{{\mathbb R}^n}\left |\left ({\mathcal I}_A^{-1}(t)( x-y+{\mathcal I}_b(t)) , \bs\ell \right )\right |^{p'}
e^ {-p'\left |{\mathcal I}_A^{-1/2}(t)( x-y+{\mathcal I}_b(t))\right |^2\big /4}dy \right \}^{1/p'}.
\end{eqnarray*}
Changing the variable $\xi={\mathcal I}_A^{-1/2}(t)( x-y+{\mathcal I}_b(t) )$ in the last integral in view of
$dy =\left (\det {\mathcal I}_A^{1/2}(t)\right )d\xi$,
we arrive at the following representation
$$
{\mathcal K}_{p, \bs\ell}(t)=\frac{\bv e^{{\mathcal I}_{C^*}(t)}\bv \left (\det {\mathcal I}_A^{1/2}(t)\right )^{1/p'}}{2(2\sqrt{\pi })^n\det {\mathcal I}_A^{1/2}(t)}
\left \{\int_{{\mathbb R}^n}\left |\left ({\mathcal I}_A^{-1/2}(t)\xi , \bs\ell \right )\right |^{p'}
e^{-p' | \xi|^2/4 }d\xi \right \}^{1/p'}.
$$
By the  symmetricity of ${\mathcal I}_A^{-1/2}(t)$, we have
\begin{equation} \label{Eq_3.11ABCD}
{\mathcal K}_{p, \bs\ell}(t)=\frac{\bv e^{{\mathcal I}_{C^*}(t)}\bv }{2(2\sqrt{\pi })^n\left (\det {\mathcal I}_A^{1/2}(t)\right )^{1/p}}
\left \{\int_{{\mathbb R}^n}\left |\left (\xi ,{\mathcal I}_A^{-1/2}(t)\bs\ell \right )\right |^{p'}
e^{-p' | \xi|^2/4 }d\xi \right \}^{1/p'}.
\end{equation}

Passing to the spherical coordinates in (\ref{Eq_3.11ABCD}), we obtain
\begin{equation} \label{Eq_3.12A}
{\mathcal K}_{p, \bs\ell}(t)\!=\!\frac{\bv e^{{\mathcal I}_{C^*}(t)}\bv}{2(2\sqrt{\pi })^n\!\!\left (\det {\mathcal I}_A^{1/2}(t)\right )^{1/p}}\!  
\left \{\int_0^{\infty}\!\!\!\!\rho^{p'+n-1}e^{-p'\rho^2 /4} d\rho\!\!\int_{{\mathbb S}^{n-1}}\! \left |
\big (\bs e_\sigma , {\mathcal I}_A^{-1/2}(t)\bs\ell \big ) \right |^{p'} \!\!\!d\sigma \!\right \}^{\!\!1/p'}\!,
\end{equation}
where $\bs e_\sigma$ is the $n$-dimensional unit vector joining the origin to a point $\sigma$ of the sphere ${\mathbb S}^{n-1}$.

Let $\vartheta$ be the angle between $\bs e_{\sigma}$ and ${\mathcal I}_A^{-1/2}(t)\bs\ell $. We have
\begin{eqnarray} \label{Eq_3.13}
& &\int_{{\mathbb S}^{n-1}} \big |\big (\bs e_{\sigma},{\mathcal I}_A^{-1/2}(t)\bs\ell \big ) \big |^{p'} d\sigma=2\omega_{n-1}
\big |{\mathcal I}_A^{-1/2}(t)\bs\ell \big |^{p'}\int_0^{\pi/2} 
\cos^{p'} \vartheta  \sin^{n-2}\vartheta d\vartheta\nonumber\\
& &\nonumber\\
& &=\omega_{n-1}\big |{\mathcal I}_A^{-1/2}(t)\bs\ell \big |^{p'} B\left ( \frac{p'+1}{2}, \frac{n-1}{2} \right )
=\big |{\mathcal I}_A^{-1/2}(t)\bs\ell \big |^{p'}\:\frac{2\pi^{(n-1)/2}\Gamma \left ( \frac{p'+1}{2} \right )}
{\Gamma \left ( \frac{n+p'}{2} \right )}\;.
\end{eqnarray}
Further, making the change of variable $\rho=\sqrt{u}$ in the integral 
$$
\int_0^{\infty}\rho^{p'+n-1} e^{-p'\rho^2/4} d\rho
$$
and applying (\ref{Eq_3.012}), we obtain
\begin{equation} \label{Eq_3.14}
\int_0^{\infty} \rho^{p'+n-1} e^{-p'\rho^2/4} d\rho=\frac{1}{2}\int_0^{\infty}u^{\frac{p'+n}{2}-1}
e^{-p'u/4 }du=
\frac{1}{2}\left ( \frac{4}{p'} \right )^{\frac{p'+n}{2}}{\Gamma \left ( \frac{n+p'}{2} \right )}\;.
\end{equation}
Combining (\ref{Eq_3.13}) and (\ref{Eq_3.14}) with (\ref{Eq_3.12A}), we arrive at (\ref{Eq_1.3P}). 

Formula (\ref{Eq_3.7}) follows from (\ref{Eq_1.3P}) and (\ref{Eq_1.3AB}).
As a particular case of (\ref{Eq_3.7}) with $p=\infty$, we obtain (\ref{Eq_1.4}).
\end{proof}

\section{Sharp estimates for solutions to the nonhomogeneous weakly coupled parabolic system} \label{S_5}

In this section we derive estimates for solution of the Cauchy problem (\ref{NH}). Here we suppose that $\bs f\in [L^p({\mathbb R}^{n+1}_T)]^m$ $\cap [C^\alpha\big (\overline{{\mathbb R}^{n+1}_T} \big )]^m$, $\alpha \in (0, 1)$. 
It follows from the known assertions for the single parabolic equation (e.g. \cite{FR}, Ch.1, Sect. 7 and 9) and  (\ref{Eq_3.02B}), (\ref{NvNH}) that formula (\ref{Eq_3.2a}), where $P$ is defined by (\ref{Eq_4.6B}), solves problem  (\ref{NH}) with $\bs f\in [L^p({\mathbb R}^{n+1}_T)]^m$ $\cap [C^\alpha\big (\overline{{\mathbb R}^{n+1}_T} \big )]^m$.

First, we prove the sharp pointwise estimate for $|\bs u|$, where $\bs u$ is solution of problem (\ref{NH}).

\begin{theorem} \label{T_03} Let $(x, t)$ be an arbitrary point in ${\mathbb R}^{n+1}_T$ and $\bs u $ be solution of problem $(\ref{NH})$. Let us suppose that the integral
$$
\int_0^t\frac{\left | e^{{\mathcal I}_{C^*}(t, \tau)}\bs z\right |^{p'}}{\big (\det {\mathcal I}_A^{1/2}(t, \tau)\big )^{p'-1}}\;d\tau
$$
is convergent. The sharp coefficient ${\mathcal N}_p(t)$ in inequality $(\ref{Eq_3.007N})$ is given by $(\ref{Eq_3.007AN})$.
As a special case of $(\ref{Eq_3.007AN})$ with $p=\infty$ one has $(\ref{Eq_3.007BN})$.
\end{theorem}
\begin{proof} By Proposition \ref{P_1.2} and (\ref{Eq_3.2a}),
(\ref{Eq_4.6B}), the sharp coefficient in inequality (\ref{Eq_3.007N}) is given by
$$ 
{\mathcal N}_p(t)=\frac{1}{(2\sqrt{\pi })^n}\max_{|\bs z|=1}
\left \{\!\int_0^t\int_{\mathbb R ^n}\!\frac{
\left |\left(e^{{\mathcal I}_C(t, \tau)}\right )^*\bs z\right |^{p'}}{\big (\det {\mathcal I}_A^{1/2}(t, \tau)\big )^{p'}}e^ {-p'\left |{\mathcal I}_A^{-1/2}(t,  \tau)( x-y+{\mathcal I}_b(t, \tau))\right |^2\big /4} dy d \tau \!\right \}^{1/p'}\!\!\!,
$$ 
which, in view of 
\begin{equation} \label{Eq_3.002ABN}
\left ( e^{{\mathcal I}_C(t, \tau)} \right )^*= e^{{\mathcal I}_{C^*}(t, \tau)},
\end{equation}
implies
\begin{equation} \label{Eq_3.002AN}
{\mathcal N}_p(t)=\frac{1}{(2\sqrt{\pi })^n}\max_{|\bs z|=1}
\left \{\!\int_0^t\int_{\mathbb R ^n}\!\frac{
\left | e^{{\mathcal I}_{C^*}(t, \tau)}\bs z\right |^{p'}}{\big (\det {\mathcal I}_A^{1/2}(t, \tau)\big )^{p'}}e^ {-p'\left |{\mathcal I}_A^{-1/2}(t,  \tau)( x-y+{\mathcal I}_b(t, \tau))\right |^2\big /4} dy d \tau \!\right \}^{1/p'}\!\!\!.
\end{equation}
Now, we introduce the new variable
$\xi={\mathcal I}_A^{-1/2}(t, \tau)( x-y+{\mathcal I}_b(t, \tau) )$.
Since $y=-{\mathcal I}_A^{1/2}(t, \tau)\xi +x+{\mathcal I}_b(t,  \tau)$, we have $dy =\det {\mathcal I}_A^{1/2}(t, \tau)d\xi$, 
which together with (\ref{Eq_3.002AN}) leads to the following representation 
\begin{equation} \label{Eq_3.010N}
{\mathcal N}_p(t)=\frac{1}{(2\sqrt{\pi })^n}\max_{|\bs z|=1}
\left \{\!\int_0^t\int_{\mathbb R ^n}\!\frac{
\left | e^{{\mathcal I}_{C^*}(t, \tau)}\bs z\right |^{p'}}{\big (\det {\mathcal I}_A^{1/2}(t, \tau)\big )^{p'-1}}e^ {-p'|\xi |^2/4} d\xi d \tau \!\right \}^{1/p'}\!\!\!.
\end{equation}
Passing to the spherical coordinates in (\ref{Eq_3.010N}), we obtain
\begin{eqnarray} \label{Eq_3.011N}
{\mathcal N}_p(t)&=&\frac{1}{(2\sqrt{\pi })^n}\max_{|\bs z|=1}
\left \{\!\int_0^t \frac{
\left | e^{{\mathcal I}_{C^*}(t, \tau)}\bs z\right |^{p'}}{\big (\det {\mathcal I}_A^{1/2}(t, \tau)\big )^{p'-1}}d \tau\int_{{\mathbb S}^{n-1}}d\sigma \int_0^{\infty} \rho^{n-1}e^{-p'\rho^2/4} d\rho \!\right \}^{1/p'}\!\!\!\nonumber\\
&=&\frac{1}{(2\sqrt{\pi })^n}\max_{|\bs z|=1}
\left \{\!\omega_n\int_0^t \frac{
\left | e^{{\mathcal I}_{C^*}(t, \tau)}\bs z\right |^{p'}}{\big (\det {\mathcal I}_A^{1/2}(t, \tau)\big )^{p'-1}}d \tau \int_0^{\infty} \rho^{n-1}e^{-p'\rho^2/4} d\rho \!\right \}^{1/p'}.
\end{eqnarray}
Substituting (\ref{Eq_3.013}) into (\ref{Eq_3.011N}), we arrive at (\ref{Eq_3.007AN}). 
As a particular case of (\ref{Eq_3.007AN}) with $p=\infty$, we obtain (\ref{Eq_3.007BN}).
\end{proof}

In the next assertion we prove the sharp pointwise estimate for $|\partial \bs u/\partial\bs \ell |$, where $\bs u$ is solution of problem (\ref{NH}).

\begin{theorem} \label{T_04} Let $(x, t)$ be an arbitrary point in ${\mathbb R}^{n+1}_T$ and let $\bs u $ solve  problem $(\ref{NH})$. 
Suppose that the integral
$$
 \int_0^t\frac{\left | {\mathcal I}_A^{-1/2}(t,\tau)\bs\ell\right |^{p'}\left | e^{{\mathcal I}_{C^*}(t, \tau)}\bs z\right |^{p'}} {\left (\det {\mathcal I}_A^{1/2}(t,\tau) \right )^{p'-1}} d\tau
$$ 
is convergent for every unit $n$-dimensional vector $\bs\ell$.
Then the sharp coefficient ${\mathcal C}_{p,\ell}(t)$ in inequality
$(\ref{Eq_1.5A})$ is given by $(\ref{Eq_1.5N})$.

As a consequence of $(\ref{Eq_1.5N})$, the sharp coefficient  
${\mathcal C}_{p}(t)$ in inequality $(\ref{Eq_1.5ABC})$ is given by
\begin{equation} \label{Eq_3.17A}
{\mathcal C}_p(t)=\!\frac{1}{\big \{ 2^{n}\pi^{(n\!+\!p\!-\!1)/2} \big \}^{1/p}}\left \{\!\! \frac{\Gamma \left (\frac{p'+1}{2} \right )} {p'^{\;(n+p')/2}} \!\!\right \}^{\!\!1/p'}\!\!\!\!\!\!\!\max_{|\bs\ell|=1}\max_{|\bs z|=1}
 \left \{\!\int_0^t\!\frac{\left | {\mathcal I}_A^{-1/2}(t,\tau)\bs\ell\right |^{p'}\!\!\left | e^{{\mathcal I}_{C^*}(t, \tau)}\bs z\right |^{p'}} {\left (\det {\mathcal I}_A^{1/2}(t,\tau) \right )^{p'-1}} d\tau
 \right \}^{\!1/p'}\!\!.
\end{equation}

As a special case of $(\ref{Eq_3.17A})$ with $p=\infty$ one has $(\ref{Eq_1.6})$.
\end{theorem}
\begin{proof} Differentiating in (\ref{Eq_4.6B}) with respect to $x_j$, $j=1,\dots , n$, we obtain
\begin{eqnarray} \label{Eq_5.1}
& &\frac{\partial}{\partial x_j}P(x, t, \tau)\nonumber\\
& &= -
\frac{e^{{\mathcal I}_C(t, \tau)}}{2(2\sqrt{\pi })^n\det {\mathcal I}_A^{1/2}(t, \tau)}
\left \{{\mathcal I}_A^{-1}(t, \tau)( x+{\mathcal I}_b(t, \tau))\right \}_j  
e^ {-\left |{\mathcal I}_A^{-1/2}(t, \tau)( x+{\mathcal I}_b(t, \tau))\right |^2\big /4},
\end{eqnarray}
which together with (\ref{Eq_3.1AAB}) and (\ref{Eq_3.2a}), leads to
\begin{eqnarray*} 
\hspace{-5mm}& &\frac{\partial \bs u}{ \partial { \bs\ell}}=(\bs\ell, \nabla_x)\bs u=\sum_{j=1}^n\int_0^t\int_{{\mathbb R}^n}\ell_j\frac{\partial}{\partial x_j}P(x-y, t, \tau)\bs f(y, \tau)dy d\tau=-
\frac{1}{2(2\sqrt{\pi })^n}\\
\hspace{-5mm}& &\!\!\!\times\int_0^t\!\int_{{\mathbb R}^n}\frac{e^{{\mathcal I}_C(t,\tau)}}{\det {\mathcal I}_A^{1/2}(t, \tau )}\left ({\mathcal I}_A^{-1}(t, \tau)( x-y+{\mathcal I}_b(t, \tau)), \bs\ell \right )
e^ {-\left |{\mathcal I}_A^{-1/2}(t, \tau)( x-y+{\mathcal I}_b(t, \tau))\right |^2\big /4}\bs f(y, \tau)dy d\tau.
\end{eqnarray*}
Applying Proposition \ref{P_1.2} to the last representation, we conclude that the sharp coefficient in  estimate (\ref{Eq_1.5A})
is given by
\begin{eqnarray*} 
\hspace{-7mm}& &{\mathcal C}_{p, \bs\ell}(t)=\frac{1}{2(2\sqrt{\pi })^n} \\
\hspace{-7mm}& &\!\times\max_{|\bs z|=1}
\!\left \{\!\int_0^t\!\!\int_{{\mathbb R}^n}\!\!\frac{\big |\big (e^{{\mathcal I}_C(t,\tau)}\big )^*\! \bs z\big |^{p'}}{\big (\!\det {\mathcal I}_A^{1/2}(t, \tau )\!\big )^{p'}}\!\left |\!\left ({\mathcal I}_A^{-1}(t, \tau)( x\!-\!y\!+\!{\mathcal I}_b(t, \tau)) , \bs\ell \right )\!\right |^{p'}\!\!
e^ {-p'\left |{\mathcal I}_A^{-1/2}(t, \tau)( x\!-\!y\!+\!{\mathcal I}_b(t, \tau))\right |^2\!\big /4}dy d\tau \!\right \}^{\!\!\frac{1}{p'}}\!\!\!,
\end{eqnarray*}
which in view of (\ref{Eq_3.002ABN}), implies
\begin{eqnarray*}
\hspace{-7mm}& &{\mathcal C}_{p, \bs\ell}(t)=\frac{1}{2(2\sqrt{\pi })^n} \\
\hspace{-7mm}& &\!\times\max_{|\bs z|=1}
\!\left \{\!\int_0^t\!\!\int_{{\mathbb R}^n}\!\!\frac{\big | e^{{\mathcal I}_{C^*}(t,\tau)}\bs z\big |^{p'}}{\big (\!\det {\mathcal I}_A^{1/2}(t, \tau )\!\big )^{p'}}\!\left |\!\left ({\mathcal I}_A^{-1}(t, \tau)( x\!-\!y\!+\!{\mathcal I}_b(t, \tau)) , \bs\ell \right )\!\right |^{p'}\!\!
e^ {-p'\left |{\mathcal I}_A^{-1/2}(t, \tau)( x\!-\!y\!+\!{\mathcal I}_b(t, \tau))\right |^2\!\big /4}dy d\tau \!\right \}^{\!\!\frac{1}{p'}}\!\!\!.
\end{eqnarray*}
Changing the variable $\xi={\mathcal I}_A^{-1/2}(t, \tau)( x-y+{\mathcal I}_b(t, \tau) )$ in the last integral in view of
$dy =\left (\det {\mathcal I}_A^{1/2}(t, \tau)\right )d\xi$,
we arrive at the following representation
$$
{\mathcal C}_{p, \bs\ell}(t)=
\frac{1}{2(2\sqrt{\pi })^n}\max_{|\bs z|=1}
\left \{\int_0^t\int_{{\mathbb R}^n}\frac{\big | e^{{\mathcal I}_{C^*}(t,\tau)}\bs z\big |^{p'}}{\big (\det {\mathcal I}_A^{1/2}(t, \tau )\big )^{p'-1}}\left |\left ({\mathcal I}_A^{-1/2}(t, \tau)\xi, \bs\ell \right )\right |^{p'}
e^ {-p'\left |\xi\right |^2/4}d\xi d\tau \right \}^{1/p'}\!\!\! .
$$
By the  symmetricity of ${\mathcal I}_A^{-1/2}(t, \tau)$, we have
$$ 
{\mathcal C}_{p, \bs\ell}(t)=
\frac{1}{2(2\sqrt{\pi })^n}\max_{|\bs z|=1}
\left \{\int_0^t\int_{{\mathbb R}^n}\frac{\big | e^{{\mathcal I}_{C^*}(t,\tau)}\bs z\big |^{p'}}{\big (\det {\mathcal I}_A^{1/2}(t, \tau )\big )^{p'-1}}\left |\left (\xi, {\mathcal I}_A^{-1/2}(t, \tau)\bs\ell \right )\right |^{p'}
e^ {-p'\left |\xi\right |^2/4}d\xi d\tau \right \}^{1/p'}\!\!\! .
$$ 
Passing to the spherical coordinates in the inner integral, we obtain
\begin{eqnarray} \label{Eq_3.003}
\hspace{-1.5cm}& &{\mathcal C}_{p, \bs\ell}(t)=
\frac{1}{2(2\sqrt{\pi })^n}\nonumber\\
\hspace{-1.5cm}& &\times\max_{|\bs z|=1}
\left \{\int_0^t\frac{\big | e^{{\mathcal I}_{C^*}(t,\tau)}\bs z\big |^{p'}}{\big (\det {\mathcal I}_A^{1/2}(t, \tau )\big )^{p'-1}}d\tau\int_0^{\infty}\!\!\!\!\rho^{p'+n-1}e^{-p'\rho^2 /4} d\rho\!\!\int_{{\mathbb S}^{n-1}}\! \left |
\big (\bs e_\sigma , {\mathcal I}_A^{-1/2}(t, \tau)\bs\ell \big ) \right |^{p'} \!\!\!d\sigma \right \}^{1/p'}\!\!\! ,
\end{eqnarray}
where $\bs e_\sigma$ is as before, the $n$-dimensional unit vector joining the origin to a point $\sigma$ of the sphere ${\mathbb S}^{n-1}$.

Let $\vartheta$ be the angle between $\bs e_{\sigma}$ and ${\mathcal I}_A^{-1/2}(t, \tau )\bs\ell $. Similarly to (\ref{Eq_3.13}), we have
\begin{equation} \label{Eq_3.004}
\int_{{\mathbb S}^{n-1}} \big |\big (\bs e_{\sigma},{\mathcal I}_A^{-1/2}(t, \tau)\bs\ell \big ) \big |^{p'} d\sigma=\big |{\mathcal I}_A^{-1/2}(t, \tau)\bs\ell \big |^{p'}\:\frac{2\pi^{(n-1)/2}\Gamma \left ( \frac{p'+1}{2} \right )}{\Gamma \left ( \frac{n+p'}{2} \right )}\;.
\end{equation}
Combining (\ref{Eq_3.14}) and (\ref{Eq_3.004}) with (\ref{Eq_3.003}), we arrive at (\ref{Eq_1.5N}). 
Equality (\ref{Eq_3.17A}) follows from (\ref{Eq_1.5N}) and (\ref{Eq_1.5AB}). Putting $p=\infty $ in (\ref{Eq_3.17A}), we arrive at (\ref{Eq_1.6}). 
\end{proof}

\medskip
{\bf Acknowledgement.} The publication has been prepared with the support of the "RUDN University Program 5-100".


\end{document}